\documentstyle[amsfonts,amssymb,amsthm,amsmath, color, 12pt]{article}

\title{On some topological games involving networks}
\author{Leandro F. Aurichi, Maddalena Bonanzinga\\ and Davide Giacopello}
\date{}

\begin{document}

\maketitle

\newtheorem{theorem}{Theorem}[section]
\newtheorem{corollary}[theorem]{Corollary}
\newtheorem{question}[theorem]{Question}
\newtheorem{example}[theorem]{Example}
\newtheorem{lemma}[theorem]{Lemma}
\newtheorem{proposition}[theorem]{Proposition}
\newtheorem{property}[theorem]{Property}
\newtheorem{definition}[theorem]{Definition}
\newtheorem{remark}[theorem]{Remark}
\newtheorem{problem}[theorem]{Problem}

\newcommand{\my}[1]{\textcolor{red}{\sf #1}}
\newcommand{\green}[1]{\textcolor{green}{\sf #1}}
\newcommand{\blu}[1]{\textcolor{blue}{\sf #1}}
\newcommand{\violet}[1]{\textcolor{violet}{\sf #1}}

\begin{abstract}
	In these notes we introduce and investigate two new games called R-nw-selective game and the M-nw-selective game. These games naturally arise from the corresponding selection principles involving networks introduced in \cite{BG}. 
\end{abstract}



\section{Introduction}
Throughout the paper we mean by \lq\lq space\rq\rq a topological Hausdorff space.\\
A family $\mathcal N$ of sets is called a network for $X$ if for every $x\in X$ and for every open neighbourhood $U$ of $x$ there exists an element $N$ of $\mathcal N$ such that $x\in N\subseteq U$; $nw(X)=min\{|{\mathcal N}|: {\mathcal N}$ is a network for $X\}$ is the network weight of $X$. A family $\mathcal P$ of open sets is called a $\pi$-base for $X$ if every nonempty open set in $X$ contains a nonempty element of $\mathcal P$; $\pi w(X)=min\{|{\mathcal P}|: {\mathcal P}$ is a $\pi$-base for $X\}$ is the $\pi$-weight of $X$. It is known that $nw(X)\leq w(X)$, where $w(X)$ denotes the weight of the space $X$, and in the class of compact Hausdorff spaces $nw(X)= w(X)$ (see \cite{Ba}). $\delta(X)=\sup\{d(Y): Y \hbox{ is a dense subset of }X\}$, where $d(X)$ denotes the density of the space $X$.\\
Recall that for $f,g\in \omega^\omega$, $f\leq^*g$ means that $f(n)\leq g(n)$ for all but finitely many $n$ (and $f\leq g$ means that $f(n)\leq g(n)$ for all $n\in\omega$). A subset $D\subseteq \omega^\omega$ is dominating if for each $g\in \omega^\omega$ there is $f\in D$ such that $g\leq^*f$. The minimal cardinality of a dominating subset of $\omega^\omega$ is denoted by ${\frak d}$. The value of ${\frak d}$ does not change if one considers the relation $\leq$ instead of $\leq^*$ \cite[Theorem 3.6]{vD}. $\cal M$ denotes the family of all meager subsets of $\Bbb R$. $cov(\cal M)$ is the minimum of the cardinalities of subfamilies ${\cal U}\subseteq {\cal M}$ such that $\bigcup{\cal U}=\Bbb R$. However, another description of the cardinal $cov(\cal M)$ is the following. $cov(\cal M)$ is the minimum cardinality of a family $F\subset \omega^\omega$ such that for every $g\in \omega^\omega$ there is $f\in F$ such that $f(n)\not =g(n)$ for all but finitely many $n\in \omega$ (see \cite{B} and also \cite[Theorem 2.4.1]{BJ}).
Thus if $F\subset \omega^\omega$ and $|F|<cov(\cal M)$, then there is $g\in \omega^\omega$ such that for every $f\in F$, $f(n) = g(n)$ for infinitely many $n\in\omega$; it is often said that $g$ \textit{guesses} $F$. Also if $\Bbb P$ is a countable poset and $\cal D$ is a family of dense sets of cardinality strictly less than $cov({\cal M})$ then there exists a generic that meets all the dense sets of the family \cite[Section 3]{BJ}.

\bigskip

In \cite{Sch1, Sch2} a systematic approach was considered to describe selection principles. Given two collections $\cal A$ and $\cal B$ of some particular topological objects on a space $X$, Scheepers introduced the following notation:
\begin{itemize}
	\item[$S_1({\cal A},{\cal B})$]: For every sequence $({\cal U}_n: n\in\omega)$ of elements of $\cal A$ there exists $U_n\in {\cal U}_n$, $n\in\omega$, such that $\{U_n:n\in\omega\}$ belongs to $\cal B$.
	\item[$S_{fin}({\cal A},{\cal B})$]: For every sequence $({\cal U}_n: n\in\omega)$ of elements of $\cal A$ there exists a finite subset ${\cal F}_n\in {\cal U}_n$, $n\in\omega$, such that $\bigcup_{n\in\omega} {\cal F}_n$ belongs to $\cal B$.
\end{itemize}

If one denotes by ${\cal O}$ the family of all open covers of a space $X$ and by ${\cal D}$ the family of all dense subsets of a space $X$, a space is said to be Rothberger if it satisfies $S_1({\cal O},{\cal O})$, Menger if it satisfies  $S_{fin}({\cal O},{\cal O})$, R-separable if it has the property $S_1({\cal D},{\cal D})$ and M-separable if has the property $S_{fin}({\cal D},{\cal D})$.\\
$\delta(X) = \omega$ for every M-separable space $X$ \cite{BBMT} and if $\delta(X) = \omega$  and 
$\pi w(X) < {\frak d}$, then $X$ is M-separable (a stronger version of this fact is estabilished in \cite[Theorem 40]{Sch2}); moreover,
if $\delta(X) = \omega$  and 
$\pi w(X) <$ cov$({\cal M})$, then $X$ is R-separable (a stronger version of this fact is estabilished in \cite[Theorem 29]{Sch2}).\\ 
Every space having a countable base is R-separable, therefore M-separable. However,  not every space with countable network weight is M-separable. Hence in \cite{BG} the authors asked under which conditions a space with countable network weight must be M-separable and introduced and studied the following classes of spaces. \begin{definition}\rm
	Let $X$ be a space with $nw(X)=\omega$.
	\begin{itemize}
		\item $X$ is M-nw-selective if for every sequence $({\cal N}_n : n\in\omega)$ of countable networks for $X$ one can select finite ${\cal F}_n\subset {\cal N}_n$, $n\in\omega$, such that $\bigcup_{n\in\omega} {\cal F}_n$ is a network for $X$.
		\item $X$ is R-nw-selective if for every sequence $({\cal N}_n : n\in\omega)$ of countable networks for $X$ one can pick  $F_n\in {\cal N}_n$, $n\in\omega$, such that $\{F_n :n\in\omega\}$ is a network for $X$.
	\end{itemize}	
\end{definition}
In \cite{BG} it was proved that any R-nw-selective (M-nw-selective) space is both Rothberger and R-separable (Menger and M-separable). See also \cite{BGZ} for more details about these two properties.\\
Recall that topological games, introduced with a systematical structure in \cite{Sch1, Sch2}, are infinite games played by two different players, {\sc Alice} and {\sc Bob}, on a topological space $X$ (see also \cite{Au}). We assume that the lenght (number of innings) of the games is $\omega$ and the two players pick in each inning some topological objects of a fixed space. The strategies of the two players are a priori defined; they are some functions that take care of the game history. At the end there is only one winner, so a draw is not allowed. Playing a game $G$ on a space $X$ gives rise to two properties: \lq\lq {\sc Alice} has a winning strategy in the game $G$ on $X$\rq\rq; \lq\lq {\sc Bob} has a winning strategy in the game $G$ on $X$\rq\rq. Of course, since there is not draw, it is impossible for a space to have both these properties, but it can happen that the negation of both of them holds. In this case we say that the game $G$ is undeterminate on the space $X$.\\
Given two families of topological objects $\cal A$ and $\cal B$, the followings are two games associated to  selection principles.
\begin{itemize}
	\item[$G_1({\cal A},{\cal B})$]: is played according to the following rules.
	\begin{itemize}
		\item for every $n\in\omega$ {\sc Alice} chooses $A_n\in{\cal A}$;
		\item {\sc Bob} answers picking $b_n\in A_n$ for each $n\in\omega$;
		\item the winner is {\sc Bob} if $\{b_n: n\in\omega\}\in {\cal B}$, otherwise {\sc Alice} wins.
	\end{itemize}
	\item[$G_{fin}({\cal A},{\cal B})$]: is played according to the following rules.
	\begin{itemize}
		\item for every $n\in\omega$ {\sc Alice} chooses $A_n\in{\cal A}$;
		\item {\sc Bob} answers picking a finite subset $B_n\subseteq A_n$ for each $n\in\omega$;
		\item the winner is {\sc Bob} if $\bigcup\{B_n: n\in\omega\}\in {\cal B}$, otherwise {\sc Alice} wins.
	\end{itemize}
\end{itemize}

The game $G_1({\cal O},{\cal O})$, called Rothberger game, is strictly related to the Rothberger property. In the following we denote this game by Rothberger$(X)$. 
The game $G_{fin}({\cal O},{\cal O})$, called Menger game, is strictly related to the Menger property. In the following we denote this game by Menger$(X)$.
The game $G_1({\cal D},{\cal D})$ is strictly related to the R-separarability, in the following we denote this game by R-separable$(X)$. 
Similarly, the game $G_{fin}({\cal D},{\cal D})$ is strictly related to the M-separability and in the following we denote this game M-separable$(X)$.\\
This games were largely studied and some important characterizations of \lq\lq {\sc Alice} does not have a winning strategy\rq\rq and \lq\lq {\sc Bob} has a winning strategy\rq\rq have been given \cite{Au, Sch2, Tel, Tel2}. Despite this some questions are still open. We denote by ${\sc Bob}\uparrow G$ on $X$, the fact that \lq\lq {\sc Bob} has a winning strategy in the game $G$ on $X$\rq\rq and by ${\sc Alice}\not\uparrow G$ on $X$, the fact that \lq\lq {\sc Alice} does not have a winning strategy in the game $G$ on $X$\rq\rq.

\begin{remark}\rm\label{remarkina}
	In general the following implications hold.
	\begin{enumerate}
		\item ${\sc Bob}\uparrow G_1({\cal A},{\cal B}) \implies {\sc Bob}\uparrow G_{fin}({\cal A},{\cal B})$;
		\item ${\sc Alice}\not\uparrow G_1({\cal A},{\cal B}) \implies {\sc Alice}\not\uparrow G_{fin}({\cal A},{\cal B})$;
		\item\label{3} ${\sc Bob}\uparrow G_1({\cal A},{\cal B}) \implies {\sc Alice}\not\uparrow G_1({\cal A},{\cal B})\implies S_1({\cal A},{\cal B})$;
		\item\label{4} ${\sc Bob}\uparrow G_{fin}({\cal A},{\cal B}) \implies {\sc Alice}\not\uparrow G_{fin}({\cal A},{\cal B})\implies S_{fin}({\cal A},{\cal B})$.
	\end{enumerate}
	For some properties the last two implications of points \ref{3} and \ref{4} are, in fact, characterizations, that is ${\sc Alice}\not\uparrow G({\cal A},{\cal B})\iff S({\cal A},{\cal B})$.
\end{remark}
In \cite{Paw} it is proved that a space $X$ is Rothberger if, and only if, ${\sc Alice} \not\uparrow$ Rothberger$(X)$. In \cite{Tel, Gal} it is proved that if $X$ is a space in which each point is a $G_\delta$, {\sc Bob} $\uparrow$ Rothberger$(X)$ if, and only if, $X$ is countable. Similar arguments are valid for the Menger case: in \cite{Sch, Sze} it is proved that a space $X$ is Menger if, and only if, {\sc Alice}$\not\uparrow$ Menger$(X)$ and in \cite{Tel2} that if $X$ is a metrizable space, {\sc Bob} $\uparrow$ Menger$(X)$ if, and only if, $X$ is $\sigma$-compact. In \cite{Sch2} it is proved that {\sc Bob}$\uparrow$ R-separable$(X)$ if, and only if, $\pi w(X)=\omega$ and, under CH, it is given an example of a R-separable space $X$ such that {\sc Alice} $\uparrow$ R-separable$(X)$. 

\bigskip
Two topological games $G$ and $G'$ are called dual if both \lq\lq ${\sc Alice}\uparrow G \iff {\sc Bob}\uparrow G'$\rq\rq and \lq\lq${\sc Alice}\uparrow G'\iff {\sc Bob}\uparrow G$\rq\rq hold. Sometimes this dual vision could be useful to apply different techniques in demonstrations. For instance, the Point-open game is the dual of the Rothberger game (see \cite{Gal}), the Point-picking game is the dual of $G_1({\cal D},{\cal D})$ (see \cite{Sch2}), the Compact-open game is a possible dual of the Menger game (see \cite{Tel2}), but the question about the hypotesis to add to let them be dual is still open. 

\bigskip
In Section \ref{sec1} we study the R-nw-selective game. We present a characterization of the \lq\lq{\sc Bob} having a winning strategy\rq\rq property and a sufficient condition for \lq\lq{\sc Alice} not having a winning\rq\rq strategy property. We also give a consistent characterization, in terms of games, of the R-nw-selective property in the class of spaces without isolated points, with countable netweight and weight strictly less than $cov({\cal M})$. Moreover, we introduce the (Point, Open)-Set game and we prove that it is a promising candidate to be the dual of the R-nw-selective game.

In Section \ref{sec2} we study the M-nw-selective game. We present, under some consistent hypotesis, a sufficient condition for \lq\lq{\sc Alice} not having a winning strategy\rq\rq property and some necessary conditions for the \lq\lq{\sc Bob} having a winning strategy\rq\rq property. We also give a consistent characterization, in terms of games, of the M-nw-selective property in the class of spaces with countable netweight and weight strictly less than $\frak d$. 

\section{The R-nw-selective game}\label{sec1}
\begin{definition}\rm
	Let $X$ be a space with $nw(X)=\omega$. The R-nw-selective game, denoted by R-nw-selective$(X)$, is played according to the following rules. {\sc Alice} chooses a countable network ${\cal  N}_0$ and {\sc Bob} answers picking an element $N_0\in{\cal N}_0$. Then {\sc Alice} chooses another countable network ${\cal N}_1$ and {\sc Bob} answers in the same way and so on for countably many innings. At the end {\sc Bob} wins if the set $\{N_n:n\in\omega\}$ of his selections is a network. 
\end{definition}

Simultaneously we consider the possible dual version of the R-nw-selective game. 

\begin{definition}\rm
	The (Point, Open)-Set game on a space $X$, denoted by PO-set$(X)$, is played according to the following rules. {\sc Alice} chooses a point $x_0$ and an open set $U_0$ containing $x_0$. Then {\sc Bob} picks $N_0$ a subset of $X$ such that $x_0\in N_0\subseteq U_0$. The game go ahead in this way for every $n\in\omega$ and {\sc Alice} wins if the set $\{N_n:n\in\omega\}$ of {\sc Bob}'s choices is a network.
\end{definition}

\begin{proposition}\rm\label{prop2}
	Let $X$ be a space. {\sc Bob} $\uparrow$ R-nw-selective$(X)$ if, and only if, the space $X$ is countable and second countable.
\end{proposition}
\begin{proof} 
		Let ${\Bbb M}$ be the collection of all countable networks of $X$. Let $\sigma$ a winning strategy for Bob.\\ 
		First we prove that the space is countable.\\
			{Claim 1.} $|\bigcap_{{\mathcal N}\in{\Bbb M}}\overline{\sigma({\mathcal N})}|\leq 1$.\\ Indeed, suppose to have two points, say $x$ and $y$, that are in all the clousure of the possible answers to $\mathcal N$, for any ${\mathcal N}\in {\Bbb M}$. Since the space is $T_1$, there exist a closed set $C_x$ containing $x$ and not $y$, and a closed set $C_y$ containing $y$ and not $x$. Pick a countable network $\mathcal N$. If $N\in\mathcal N$ is such that $x\in N$, then $C_x\cap N\subset C_x$ and if $N\in\mathcal N$ is such that $y\in N$, then $C_y\cap N\subset C_y$; then it is possible to construct a countable network such that any elements  of it contains at most one of the points.\\
		{Claim 2.} There exists a countable ${\Bbb M}'\subset {\Bbb M}$ such that $\bigcap_{{\mathcal N}\in{\Bbb M}'}\overline{\sigma({\mathcal N})
		}=\bigcap_{{\mathcal N}\in{\Bbb M}}\overline{\sigma({\cal N})}$.\\
	 Indeed, if $\bigcap_{{\mathcal N}\in{\Bbb M}}\overline{\sigma({\cal N})}=\{x\}$ (it is the same if $\bigcap_{{\mathcal N}\in{\Bbb M}}\overline{\sigma({\cal N})}=\emptyset$), the complements of all the closures form an open cover of $X\setminus\{x\}$ (or $X$) and then, since having countable network implies hereditary Lindel\"ofness, we can obtain a countable subcover of $X\setminus \{x\}$ (or of the all space $X$).\\
		Note that it is straightforward to prove the claims 1. and 2.  for any inning. Consider the following tree of possible evolution of the R-nw-selective game on $X$.
		By claim 2. there exists $({\mathcal N}_{\emptyset}^n)_{n\in\omega}$, that is countably many possible choices of {\sc Alice} in the first inning ${\mathcal N}_{\emptyset}^0, {\mathcal N}_{\emptyset}^1,  {\mathcal N}_{\emptyset}^2, ...$, such that 
		$\bigcap_{{\mathcal N}\in{\Bbb M}}\overline{\sigma({\mathcal N})}
		=\bigcap_{n\in\omega}\overline{\sigma({\mathcal N}_{\emptyset}^n)}$.\\
		Fix, for example, the branch with ${\mathcal N}^0_{\emptyset}$ then there exists a sequence $({\mathcal N}_{<0>}^n)_{n\in\omega}$ such that 
		$\bigcap_{{\mathcal N}\in{\Bbb M}}\overline{\sigma({\mathcal N}^0_{\emptyset}, {\mathcal N})}
		=\bigcap_{n\in\omega}\overline{\sigma({\mathcal N}^0_{\emptyset}, {\mathcal N}_{<0>}^n)}$.\\
		Again consider, for example, ${\mathcal N}^1_{<0>}$, then there exists a sequence $({\mathcal N}_{<0,1>}^n)_{n\in\omega}$ such that $\bigcap_{n\in\omega}\overline{\sigma({\mathcal N}^0_{\emptyset}, {\mathcal N}_{<0>}^{1}, {\mathcal N}_{<0,1>}^n)}=\bigcap_{{\mathcal N}\in{\Bbb M}}\overline{\sigma({\mathcal N}^0_{\emptyset}, {\mathcal N}_{<0>}^{1}, {\mathcal N})}$.
		By Claim 1, each intersection is empty or contains only one element. If the intersection $\bigcap_{n\in\omega}\overline{\sigma({\mathcal N}^n_{\emptyset})}$ is non-empty we call this element $x^\emptyset$, otherwise we go on; if $\bigcap_{n\in\omega}\overline{\sigma({\mathcal N}^0_{\emptyset}, {\mathcal N}_{<0>}^n)}$ is not empty we call this element $x^{<0>}$;
		if the intersection $\bigcap_{n\in\omega}\overline{\sigma({\mathcal N}^0_{\emptyset}, {\mathcal N}_{<0>}^{1}, {\mathcal N}_{<0,1>}^n)}$ is not empty we call this element $x^{<0,1>}$, and so on.
		We obtain a subset $X_0=\{x^s: s\in \omega^{<\omega}\}$ and now we want to prove that $X_0=X$. By contradiction, assume there exists $y\in X\setminus X_0$. Then $y\notin \bigcap_{n\in\omega}\overline{\sigma({\mathcal N}^n_{\emptyset})}$; hence there exists an element of the sequence $\{\sigma({\mathcal N}_\emptyset^n): n\in\omega\}$, say $\sigma({\mathcal N}_{\emptyset}^{k_0})$, such that $y$ does not belong to it. By hypotesis, $y\notin \bigcap_{n\in\omega}\overline{\sigma({\mathcal N}^{k_0}_{\emptyset}, {\mathcal N}_{<k_0>}^n)}$; hence there exists an element of $\{\sigma({\mathcal N}_{<k_0>}^n): n\in\omega\}$, say $\sigma({\mathcal N}_{<k_0>}^{k_1})$, such that $y$ does not belong to it. Again,  
		$y\not \in \bigcap_{n\in\omega}\overline{\sigma({\mathcal N}^{k_0}_{\emptyset}, {\mathcal N}_{<k_0>}^{k_1}, {\mathcal N}_{<k_0,k_1>}^n)}$,
		there exists an element of $\{\sigma({\mathcal N}_{<k_0,k_1>}^n): n\in\omega\}$, say $\sigma({\mathcal N}_{<k_0,k_1>}^{k_2})$, such that $y$ does not belong to it. Proceeding in this way we obtain a branch consisting of elements that do not contain $y$; a contradiction, because such a branch is a network because $\sigma$ is a winning strategy for {\sc Bob}. Then $X$ is countable.\\
		Now we prove that $X$ is secound countable.\\
		{Claim 3.} If $\bigcap_{{\mathcal N}\in{\Bbb M}}\overline{\sigma({\mathcal N})}=\{x\}$, there exists an open set $V$ such that $x\in V\subset \bigcup_{{\mathcal N}\in{\Bbb M}}\overline{\sigma({\mathcal N})}$.\\
		Indeed, assume by contradiction that for every open set $V$ such that $x\in V$ there exists $y_V\in V\setminus \overline{\sigma({\mathcal N})}$, for every ${\mathcal N}\in {\Bbb M}$. Let ${\mathcal N}$ be a countable  network and consider the family ${\mathcal N}^\prime=({\mathcal N}\setminus {\mathcal N}_x)\cup \{(x,y_V): V\in\tau_x\}$, where $\tau_x$ denotes the family of all open sets containing $x$ and ${\mathcal N}_x=\{N\in{\mathcal N}: x\in \overline{N}\}$.
		Since $X$ is countable, ${\mathcal N}^\prime$ is countable. Now we prove that ${\mathcal N}^\prime$ is a network. Let  $y\in X$, $y\neq x$, and $A$ be an open set such that $y\in A$. Since $X$ is $T_2$, there exists an open set $B$ such that $y\in B$ and $x\notin \overline{B}$. Then there exists $N\in{\mathcal N}$ such that $y\in N\subset A\cap B$. Therefore $N\in {\mathcal N}\setminus {\mathcal N}_x$.
		\\
		{Claim 4.} If $\bigcap_{{\mathcal N}\in{\Bbb M}}\overline{\sigma({\mathcal N})}=\{x\}$, 
		there exists ${\Bbb M}^\prime\subset{\Bbb M}$ countable such that 
		$\bigcap_{{\mathcal N}\in{\Bbb M}^\prime}\overline{\sigma({\mathcal N})}=\{x\}$ and also such that $\bigcup_{{\mathcal N}\in{\Bbb M}^\prime}\overline{\sigma({\mathcal N})}=\bigcup_{{\mathcal N}\in{\Bbb M}}\overline{\sigma({\mathcal N})}$.\\ 
		Recall that, by Claim 2 there exists a countable subset ${\Bbb M}^*\subset {\Bbb M}$  such that $\bigcap_{{\mathcal N}\in{\Bbb M}}\overline{\sigma({\mathcal N})}=\bigcap_{{\mathcal N}\in{\Bbb M}^*}\overline{\sigma({\mathcal N})}$; further, since $X$ is countable,  $\bigcup_{{\mathcal N}\in{\Bbb M}}\overline{\sigma({\mathcal N})}$ is countable and then there exists a countable subset ${\Bbb M}^{**}\subset {\Bbb M}$ such that $\bigcup_{{\mathcal N}\in{\Bbb M}^{**}}\overline{\sigma({\mathcal N})}=\bigcup_{{\mathcal N}\in{\Bbb M}}\overline{\sigma({\mathcal N})}$. Then ${\Bbb M}^\prime={\Bbb M}^*\cup{\Bbb M}^{**}$.\\
		Even Claims 3 and 4 are clearly true for each inning.
		Consider the construction of the tree in the previous part of the proof. We know that $|\bigcap_{n\in\omega}\overline{\sigma({\mathcal N}^n_{\emptyset})}|\leq 1$. If $\bigcap_{n\in\omega}\overline{\sigma({\mathcal N}^n_{\emptyset})}\neq\emptyset$, fix $V_{\emptyset}$ as in Claim 3. If $\bigcap_{k\in\omega}\overline{\sigma({\mathcal N}^n_{\emptyset}, {\mathcal N}^k_{<n>})}\neq\emptyset$, fix $V_{<n>}$ as in Claim 3 and so on. Now we prove that $\{V_s: s\in \omega^{<\omega}\}$ is a base. Assuming the contrary, there exist $x\in X$ and an open set $A$ with $x\in A$ such that for every $s\in \omega^{<\omega}$ such that $x\in V_s$, $V_s$ is not contained in $A$. 
		In the first inning, we have a family ${\Bbb M}^\prime$ of countably many networks obtained as in Claim 4. Consider the intersection $\bigcap_{N\in{\Bbb M}^\prime}\overline{\sigma({\mathcal N})}$.
		If $\bigcap_{N\in{\Bbb M}^\prime}\overline{\sigma({\mathcal N})}=\emptyset$, we can pick a ${\mathcal N}\in{\Bbb M}^\prime$, such that $x\notin \sigma({\mathcal N})$. If $\bigcap_{N\in{\Bbb M}^\prime}\overline{\sigma({\mathcal N})}=\{y\}$ we have two cases: if $y\neq x$, we can pick a ${\mathcal N}\in{\Bbb M}^\prime$, such that $x\notin \sigma({\mathcal N})$; if $y=x$, then we can pick, if there exists a ${\mathcal N}\in{\Bbb M}^\prime$, such that $x\not\in \sigma({\mathcal N})$, otherwise by hypothesis and Claim 3 we can pick a ${\mathcal N}\in{\Bbb M}^\prime$, such that $\overline{\sigma({\mathcal N})}$ is not contained in $A$. Then, proceeding in this way for each inning, we find a branch of the tree, i.e., a R-nw-selective$(X)$ in which {\sc Alice} has a winning strategy, a contradiction. 
	\end{proof}

The following proposition shows that the (Point, Open)-set game is a good candidate to be the dual of the R-nw-selective game. 
\begin{proposition}\rm
	Let $X$ be a space. The following implications hold.
	\begin{enumerate} 
		\item {\sc Alice} $\uparrow $ PO-set$(X) \implies$  {\sc Bob} $\uparrow $ R-nw-selective$(X)$.
		\item {\sc Alice} $\uparrow $ R-nw-selective$(X)\implies $ {\sc Bob} $\uparrow $ PO-set$(X)$.
		\item {\sc Bob} $\uparrow $ R-nw-selective$(X)\implies$  {\sc Alice} $\uparrow $ PO-set$(X)$.
		
	\end{enumerate}
\end{proposition}
\begin{proof}
	The proof of points 1. and 2. is trivial. By Proposition \ref{prop2} it is straightforward to prove the point 3. 	
\end{proof}
\begin{question}\rm
	Does {\sc Bob} $\uparrow $ PO-set$(X)$ imply {\sc Alice} $\uparrow $ R-nw-selective$(X)$?
\end{question}
Now we study the determinacy of the R-nw-selective game.
\begin{proposition}\rm\label{prop1}
	Let $X$ be a space with $nw(X)=\omega$. If $|X|<cov({\cal M})$ and $w(X)<cov({\cal M})$, then {\sc Alice} $\not\uparrow$ R-nw-selective$(X)$. 
\end{proposition}
\begin{proof}
	Suppose, by contradiction, that $\sigma$ is a winning strategy for {\sc Alice} in the R-nw-selective$(X)$ and fix a base $\cal B$ of cardinality $w(X)$. Construct a countable tree using the strategy $\sigma$ in such a way that $\sigma(\langle\rangle)={\cal N}_0$; for each $N_0\in {\cal N}_0$ apply the strategy and so on.
	Look at this tree as the poset of all finite branches ordered with the inverse natural extention. The nodes in this tree are the countable networks that are images through the function $\sigma$. Fix $x\in X$ and $B\in {\cal B}$ containing $x$, the set $D_{(x,B)}$ of all the finite sequences of the tree such that there exists an element of the sequence that is a $\sigma(\langle ..., N\rangle)$ with $x\in N\subset B$,
	is dense in the poset. Since the cardinality of the family $\{D_{(x,B)}: x\in X \hbox{ and } B\in {\cal B}\}$ is less than $cov({\cal M})$ there exists a generic filter whose union is a branch of the tree and that intersects all the dense sets of the family. Therefore we obtain the contradiction, in particular this branch is a R-nw-selective$(X)$ in which {\sc Bob} wins.
\end{proof}

\begin{example}\rm 
	($\omega_1<cov({\cal M})$) Consider a subspace $X\subset {\Bbb R}$ with cardinality $\omega_1$. By Proposition \ref{prop2} and Proposition \ref{prop1} the R-nw-selective game on $X$ is indeterminate.
\end{example}
\begin{question}\rm
	Is there any ZFC example of a space in which the R-nw-selective game turns out to be indeterminate?
\end{question}

The following diagram shows all the relations found above.
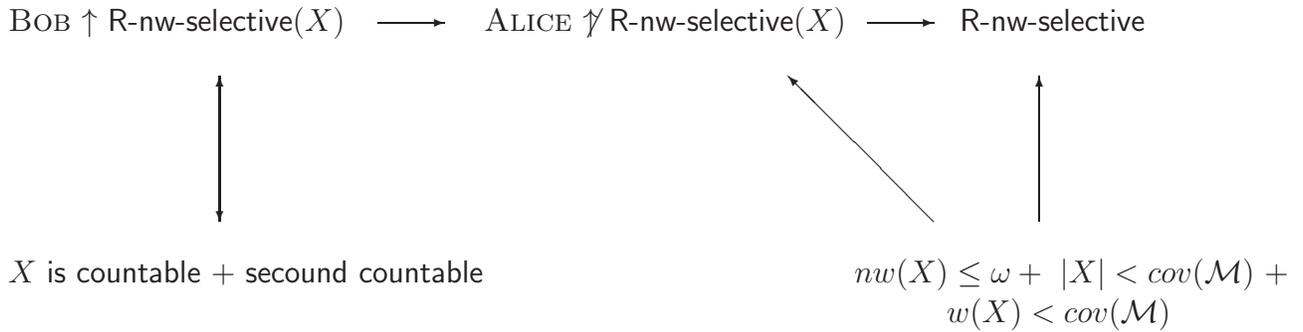
\begin{flushleft}
	\vspace{0.5cm}
	\begin{picture}(300,100) 
	\put(130,90){\textsf{{{\sc Alice} $\not\uparrow$ R-nw-selective$(X)$}}}
	
	\put(-50,-5){\textsf{{$X$ is countable + secound countable}}} 
	\put(30,73){\vector(0,-1){55}}
	\put(30,18){\vector(0,1){55}}
	\put(-50,90){\textsf{{{\sc Bob} $\uparrow$ R-nw-selective$(X)$}}}
	\put(90,93){\vector(1,0){25}}
	\put(275,93){\vector(1,0){25}}
	\put(270,-5){\textsf{{$nw(X)\leq\omega+ \,\, |X|<cov({\cal M})$ +}}}
	\put(305,-20){\textsf{{$w(X)<cov({\cal M})$}}} 
	\put(310,90){\textsf{{R-nw-selective}}}
	\put(340,18){\vector(0,1){55}}
	\put(300,18){\vector(-1,1){55}}
	\end{picture}
\end{flushleft}
\vspace{0.5cm}
\begin{question}\rm\label{q1}
	Does R-nw-selective property on a space $X$ imply {\sc Alice} $\not\uparrow$ R-nw-selective$(X)$?
\end{question}

Recall the following result.
\begin{proposition}\rm\cite{BGZ}
	Let $X$ be a space without isolated points, with $nw(X)=\omega$ and $w(X)<cov({\cal M})$. Then the following are equivalent.
	\begin{enumerate}
		\item $|X|<cov({\cal M})$;
		\item $X$ is R-nw-selective.
		\end{enumerate}
    \end{proposition}

Then it is possible to give a partial answer to Question \ref{q1}.

\begin{proposition}\rm
	Let $X$ be a space without isolated points, with $nw(X)=\omega$ and $w(X)<cov({\cal M})$. Then the following are equivalent.
	\begin{enumerate}
		\item $|X|<cov({\cal M})$;
		\item {\sc Alice} $ \not\uparrow$ R-nw-selective$(X)$;
		\item $X$ is R-nw-selective.
	\end{enumerate}
\end{proposition}

\bigskip
Now we show that if {\sc Bob} is forced to select a fixed number of element from each network the obtained game is equivalent to the R-nw-selective game.
Let ${\sc Nw}$ denote the class of all countable networks of a fixed space $X$. Let $k\in \omega$ and $G_k({\sc Nw}, {\sc Nw})$ on $X$ be the game played in the following way: {\sc Alice} chooses a countable network ${\cal  N}_0$ and {\sc Bob} answers picking a subset ${\cal F}_0\subset {\cal N}_0$ such that $|{\cal F}_0|=k$. Then {\sc Alice} chooses another countable network ${\cal N}_1$ and {\sc Bob} answers picking a subset ${\cal F}_1\subset {\cal N}_1$ such that $|{\cal F}_1|=k$ and so on for countably many innings. At the end {\sc Bob} wins if the set $\bigcup \{{\cal F}_n:n\in\omega\}$ of his selections is a network.

\begin{proposition}\rm
	{\sc Alice} $\uparrow$ R-nw-selective($X$) if, and only if, {\sc Alice} $\uparrow G_k({\sc Nw},{\sc Nw})$ on $X$.
\end{proposition}
\begin{proof}
	We only prove that if {\sc Alice} $\uparrow$ R-nw-selective($X$) then {\sc Alice} $\uparrow G_k({\sc Nw},{\sc Nw})$ on $X$. Suppose that {\sc Alice} selects a network ${\cal N}_0$ in the $G_k({\sc Nw},{\sc Nw})$. Then {\sc Bob} answers by picking ${\cal F}_0=\{F_1^0,...,F^0_k\}$. Let {\sc Alice} play the network ${\cal N}_0$ in the R-nw-selective$(X)$ for $k$-many innings and assume that {\sc Bob} chooses $F_i^0$, $i=1,...,k$, in the $i$-th inning. Now suppose that {\sc Alice} selects a network ${\cal N}_1$ in the $G_k({\sc Nw},{\sc Nw})$. Then {\sc Bob} answers by picking ${\cal F}_1=\{F_1^1,...,F^1_k\}$. Let {\sc Alice} play the network ${\cal N}_1$ in the R-nw-selective$(X)$ for $k$-many innings and assume that {\sc Bob} chooses $F_i^1$, $i=1,...,k$, in the $(k+i)$-th inning and so on. Then by hypotesis $\bigcup_{n\in\omega}{\cal F}_n$ is not a network. 
	\end{proof}

\begin{proposition}\rm
	{\sc Bob} $\uparrow$ R-nw-selective($X$) if, and only if, {\sc Bob} $\uparrow G_k({\sc Nw},{\sc Nw})$ on $X$.
\end{proposition}
\begin{proof}
	We prove that {\sc Bob} $\uparrow G_k({\sc Nw},{\sc Nw})$ on $X$ implies that the space $X$ is countable and secound countable. In fact the proof is similar to the one of Proposition \ref{prop2}. Let $\sigma$ be a winning strategy for {\sc Bob} in the $G_k({\sc Nw},{\sc Nw})$ on $X$ and ${\Bbb M}$ be the collection of all countable networks of the space $X$. We just need to prove the following claims.\\
	{Claim 1.} 
	$|\bigcap_{{\mathcal N}\in{\Bbb M}}\overline{\bigcup\sigma({\cal N})}|\leq k$.\\
	 We prove the claim for $k=2$. Suppose, by contradiction, that there are three different points $x_1,x_2,x_3\in\bigcap_{{\mathcal N}\in{\Bbb M}}\overline{\bigcup\sigma({\cal N})}$. Since the space is $T_1$, there exist a closed set $C_1$ containing $x_1$ but not $x_2,x_3$, a closed set $C_2$ containing $x_2$ but not $x_1,x_3$ and a closed set $C_3$ containing $x_3$ but not $x_1,x_2$. Pick a countable network $\mathcal N$. For every $N\in\mathcal N$ such that $x_1,x_2,x_3\in N$, replace $N$ with $N\cap C_i$, $i=1,2,3$. In this way we construct a countable network such that any elements of it contains at most one of the points $x_1,x_2,x_3$.\\
	{Claim 2.} There exists ${\Bbb M}'\subset {\Bbb M}$ countable such that $\bigcap_{{\mathcal N}\in{\Bbb M}'}\overline{\bigcup\sigma({\mathcal N})
	}=\bigcap_{{\mathcal N}\in{\Bbb M}}\overline{\bigcup\sigma({\cal N})}$.\\
	{Claim 3.} If $\bigcap_{{\mathcal N}\in{\Bbb M}}\overline{\bigcup\sigma({\mathcal N})}=\{x\}$, there exists an open set $V$ such that $x\in V\subset \bigcup_{{\mathcal N}\in{\Bbb M}}\overline{\bigcup\sigma({\mathcal N})}$.\\	
	{Claim 4.} If $\bigcap_{{\mathcal N}\in{\Bbb M}}\overline{\bigcup\sigma({\mathcal N})}=\{x\}$, 
	there exists ${\Bbb M}^\prime\subset{\Bbb M}$ countable such that 
	$\bigcap_{{\mathcal N}\in{\Bbb M}^\prime}\overline{\bigcup\sigma({\mathcal N})}=\{x\}$ and also such that $\bigcup_{{\mathcal N}\in{\Bbb M}^\prime}\overline{\bigcup\sigma({\mathcal N})}=\bigcup_{{\mathcal N}\in{\Bbb M}}\overline{\bigcup\sigma({\mathcal N})}$.\\ 
    The proof of Claims 2, 3 and 4 are similar to the ones in Proposition \ref{prop2}.
\end{proof}

\section{M-nw-selective game}\label{sec2}
\begin{definition}\rm
	Let $X$ be a space with $nw(X)=\omega$. The M-nw-selective game, denoted by M-nw-selective$(X)$, is played according to then following rules. {\sc Alice} chooses a countable network ${\cal  N}_0$ and {\sc Bob} answers picking a finite subset ${\cal F}_0\subset {\cal N}_0$. Then {\sc Alice} chooses another countable network ${\cal N}_1$ and {\sc Bob} answers in the same way and so on for countably many innings. At the end {\sc Bob} wins if the set $\bigcup \{{\cal F}_n:n\in\omega\}$ of his selections is a network. 
\end{definition}
\begin{proposition}\rm \label{Alice}
	$(MA[{\frak d}])$ Let $X$ be a space with $nw(X)=\omega$. If $|X|<\frak d$ and $w(X)<\frak d$, then {\sc Alice} $\not\uparrow$ M-nw-selective$(X)$. 
\end{proposition}
\begin{proof}
	Similar to the proof of Proposition \ref{prop1}.
\end{proof} 
Recall the following result.
\begin{proposition}\rm\cite{BGZ}
	Let $X$ be a space with $nw(X)=\omega$, $w(X)<{\frak d}$. Then the following conditions are equivalent:
	\begin{enumerate}
		\item $|X|<{\frak d}$;
		\item $X$ is M-nw-selective.
	\end{enumerate}
\end{proposition}
Then it is possible to give the following equivalences.
\begin{proposition}\rm
	$(MA[{\frak d}])$ Let $X$ be a space with $nw(X)= \omega$ and $w(X) < {\frak d}$. The following are equivalent:
	\begin{enumerate}
		\item $|X|< {\frak d}$;
		\item {\sc Alice} $\not\uparrow$ M-nw-selective$(X)$;
		\item $X$ is M-nw-selective.
	\end{enumerate}
\end{proposition}
However, it is worthwhile to pose the following question.
\begin{question}\rm
	Does the M-nw-selective property imply that {\sc Alice} $\not\uparrow$ M-nw-selective$(X)$?
\end{question}

\begin{proposition}\rm\label{sigma}
	Let $X$ be a regular space such that {\sc Bob }$\uparrow$ M-nw-selective(X). Then $X$ is $\sigma$-compact.
\end{proposition}
\begin{proof}
	Let $\Bbb M$ be the collection of all countable networks of $X$ and $\sigma$ a winning strategy for {\sc Bob} in M-nw-selective(X).\\
	{Claim 1:} $\bigcap_{{\cal N}\in {\Bbb M}}\overline{\bigcup \sigma({\cal N})}$ is compact.\\
	Indeed, put $K=\bigcap_{{\cal N}\in {\Bbb M}}\overline{\bigcup \sigma({\cal N})}$, let $\cal U$ be a cover made by open sets of $X$ and ${\cal N}\in {\Bbb M}$. Consider the network ${\cal N}'= \{N\in {\cal N}: \overline{N}\subset U \hbox{ for some } U\in {\cal U}\}\cup \{N\in {\cal N}: \overline{N}\cap K=\emptyset \}$. Then $K\subset \sigma(\langle {\cal N}'\rangle)$ and considering the corresponding open sets we have the compactness of $K$.\\
	{Claim 2:} There exists a countable subset ${\Bbb M}'\subset {\Bbb M}$ such that $\bigcap_{{\cal N}\in {\Bbb M}'}\overline{\bigcup \sigma({\cal N})}=\bigcap_{{\cal N}\in {\Bbb M}}\overline{\bigcup \sigma({\cal N})}$.\\
	The proof is similar to the one of {Claim 2.} in Proposition \ref{prop2} and, as in there, these claims are true also for all the other innings.\\
	There exists $({\mathcal N}_{\emptyset}^n)_{n\in\omega}$, that is countably many possible first inning ${\mathcal N}_{\emptyset}^0, {\mathcal N}_{\emptyset}^1,  {\mathcal N}_{\emptyset}^2, ...$, such that 
	$\bigcap_{{\mathcal N}\in{\Bbb M}}\overline{\bigcup\sigma({\mathcal N})}
	=\bigcap_{n\in\omega}\overline{\bigcup\sigma({\mathcal N}_{\emptyset}^n)}$.
	\\
	If Alice chooses ${\mathcal N}^0_{\emptyset}$, we can find $({\mathcal N}_{<0>}^n)_{n\in\omega}$ such that 
	$\bigcap_{{\mathcal N}\in{\Bbb M}}\overline{\bigcup\sigma({\mathcal N}^0_{\emptyset}, {\mathcal N})}
	=\bigcap_{n\in\omega}\overline{\bigcup\sigma({\mathcal N}^0_{\emptyset}, {\mathcal N}_{<0>}^n)}$.
	\\
	If then Alice chooses ${\mathcal N}^1_{<0>}$, we can find
	$({\mathcal N}_{<0,1>}^n)_{n\in\omega}$ 
	such that
	$\bigcap_{n\in\omega}\overline{\bigcup\sigma({\mathcal N}^0_{\emptyset}, {\mathcal N}_{<0>}^{1}, {\mathcal N}_{<0,1>}^n)}=\bigcap_{{\mathcal N}\in{\Bbb M}}\overline{\bigcup\sigma({\mathcal N}^0_{\emptyset}, {\mathcal N}_{<0>}^{1}, {\mathcal N})}$.
	By claims 1 and 2, each intersection, if it is not empty, is a compact subset. If the intersection is empty, we do not do anything and if $\bigcap_{n\in\omega}\overline{\bigcup\sigma({\mathcal N}^n_{\emptyset})}$ is a compact, we call this subset $K^\emptyset$. 
	If $\bigcap_{n\in\omega}\overline{\bigcup\sigma({\mathcal N}^0_{\emptyset}, {\mathcal N}_{<0>}^n)}$ is a compact subset, we call it $K^{<0>}$. 
	If $\bigcap_{n\in\omega}\overline{\bigcup\sigma({\mathcal N}^0_{\emptyset}, {\mathcal N}_{<0>}^{1}, {\mathcal N}_{<0,1>}^n)}$ is a compact subset, we call this element $K^{<0,1>}$, and so on.
	Consider the set $X_0=\bigcup\{K^s: s\in \omega^{<\omega}\}$. Now we prove that $X_0=X$. By contradiction, assume there exists $y\in X\setminus X_0$. Then $y\notin \bigcap_{n\in\omega}\overline{\bigcup\sigma({\mathcal N}^n_{\emptyset})}$; hence there exists $n_0\in\omega$ such that $y\not\in \overline{\bigcup\sigma({\mathcal N}^{n_{0}}_{\emptyset})}$. Again, $y\notin \bigcap_{n\in\omega}\overline{\bigcup\sigma({\mathcal N}^{n_0}_{\emptyset}, {\mathcal N}^{n}_{\langle n_0 \rangle} )}$; hence there exists $n_1\in\omega$ such that $y\not\in \overline{\bigcup\sigma({\mathcal N}^{n_0}_{\emptyset}, {\mathcal N}^{n_1}_{\langle n_0 \rangle} )}$.
	Proceeding in this way we obtain a branch (or an evolution of the M-nw-selective($X$)) in which {\sc Bob} does not win, a contradiction, because $\sigma$ is a winning strategy. Then $X$ is $\sigma$-compact.
\end{proof}
\begin{corollary}\rm\label{secoundcount}
	Let $X$ be a regular space in which {\sc Bob} $ \uparrow$ M-nw-selective$(X)$. Then $X$ is secound countable.
\end{corollary}
Recall that every regular secound countable space is metrizable and a space is called $\sigma$-(metrizable compact) if it is union of countably many metrizable compact spaces. Then it is possible to obtain the following corollary.

\begin{corollary}\rm \label{metri}
	Let $X$ be a regular space in which {\sc Bob} $ \uparrow$ M-nw-selective$(X)$. Then $X$ is metrizable.
\end{corollary}

\begin{corollary}\rm 
	Let $X$ be a regular space in which {\sc Bob} $ \uparrow$ M-nw-selective$(X)$. Then $X$ is $\sigma$-(metrizable compact).
\end{corollary}

The following is a consistent example showing that the M-nw-selective game can be indeterminate. It also shows that the vice versa of the Corollaries \ref{secoundcount} and \ref{metri} does not hold.

\begin{example}\rm $(MA[{\frak d}]+ \omega_1<{\frak d})$
	Consider a subset $X$ of the irrational numbers having cardinality $\omega_1$. By Proposition \ref{Alice}, {\sc Alice} $\not\uparrow$ M-nw-selective$(X)$. Since $X$ is not $\sigma$-compact, by Proposition \ref{sigma} we have that {\sc Bob} $\not \uparrow$ M-nw-selective$(X)$. 
\end{example}
The hypotesis of regularity in Corollary \ref{secoundcount} can be replaced by countability of the space as the following shows.
\begin{proposition}\rm\label{secound}
	If $X$ is a countable space in which {\sc Bob} $\uparrow$ M-nw-selective$(X)$. Then $X$ is secound countable.
\end{proposition}
\begin{proof}
	Similar to the the proof of Proposition \ref{prop2} replacing {Claims 3} and {4} with following.\\
	{Claim $3^\prime$.} If $\bigcap_{{\mathcal N}\in{\Bbb M}}\overline{\bigcup\sigma({\mathcal N})}=\{x\}$, there exists an open set $V$ such that $x\in V\subset \bigcup_{{\mathcal N}\in{\Bbb M}}\overline{\bigcup\sigma({\mathcal N})}$.\\
	{Claim $4^\prime$.} If $\bigcap_{{\mathcal N}\in{\Bbb M}}\overline{\bigcup\sigma({\mathcal N})}=\{x\}$, 
	there exists ${\Bbb M}^\prime\subset{\Bbb M}$ countable such that 
	$\bigcap_{{\mathcal N}\in{\Bbb M}^\prime}\overline{\bigcup\sigma({\mathcal N})}=\{x\}$ and also such that $\bigcup_{{\mathcal N}\in{\Bbb M}^\prime}\overline{\bigcup\sigma({\mathcal N})}=\bigcup_{{\mathcal N}\in{\Bbb M}}\overline{\bigcup\sigma({\mathcal N})}$.
\end{proof}


  \end{document}